\documentclass{amsart}
\usepackage[all]{xy}
\usepackage{color}
\usepackage{amsthm}
\usepackage{amssymb}
\usepackage[colorlinks=true]{hyperref}

\numberwithin{equation}{section}

\newtheorem{theorem}[equation]{Theorem}
\newtheorem*{theorem*}{Theorem}
\newtheorem{lemma}[equation]{Lemma}

\newtheorem*{conjecture*}{Mamma Conjecture}
\newtheorem*{conjecture1*}{Mamma Conjecture (revisited)}
\newtheorem{proposition}[equation]{Proposition}

\newtheorem*{corollary*}{Corollary}

\theoremstyle{remark}

\theoremstyle{remark}
\newtheorem{remark}[equation]{Remark}

\newcommand{\cA}{{\mathcal A}}
\newcommand{\cB}{{\mathcal B}}
\newcommand{\cC}{{\mathcal C}}
\newcommand{\cD}{{\mathcal D}}
\newcommand{\cE}{{\mathcal E}}

\newcommand{\cG}{{\mathcal G}}

\newcommand{\cK}{{\mathcal K}}

\newcommand{\cN}{{\mathcal N}}
\newcommand{\cO}{{\mathcal O}}

\newcommand{\cQ}{{\mathcal Q}}

\newcommand{\cT}{{\mathcal T}}

\newcommand{\cX}{{\mathcal X}}

\newcommand{\cZ}{{\mathcal Z}}

\newcommand{\bbA}{\mathbb{A}}
\newcommand{\bbB}{\mathbb{B}}
\newcommand{\bbC}{\mathbb{C}}

\newcommand{\bbP}{\mathbb{P}}

\newcommand{\bbQ}{\mathbb{Q}}
\newcommand{\bbZ}{\mathbb{Z}}

\DeclareMathOperator{\NChow}{NChow} 
\DeclareMathOperator{\NNum}{NNum} 
\DeclareMathOperator{\NVoev}{NVoev} 
\DeclareMathOperator{\Voev}{Voev} 
\DeclareMathOperator{\Chow}{Chow} 
\DeclareMathOperator{\Num}{Num} 

\newcommand{\onil}{\otimes_\mathrm{nil}}

\newcommand{\dgcat}{\mathrm{dgcat}}

\newcommand{\perf}{\mathrm{perf}}

\newcommand{\dg}{\mathrm{dg}}

\newcommand{\Hom}{\mathrm{Hom}}

\newcommand{\Ho}{\mathrm{Ho}}

\newcommand{\op}{\mathrm{op}}

\newcommand{\too}{\longrightarrow}

\newcommand{\ie}{\textsl{i.e.}\ }

\title[Some remarks concerning Voevodsky's nilpotence conjecture]{Some remarks concerning \\Voevodsky's nilpotence conjecture}
\author{Marcello Bernardara, Matilde Marcolli and Gon{\c c}alo~Tabuada}

\address{Marcello Bernardara, Institut de Math\'ematiques de Toulouse \\ %
Universit\'e Paul Sabatier \\ %
118 route de Narbonne \\ %
31062 Toulouse Cedex 9\\ %
France}
\email{marcello.bernardara@math.univ-toulouse.fr} 
\urladdr{http://www.math.univ-toulouse.fr/~mbernard/}

\address{Matilde Marcolli, Mathematics Department, Mail Code 253-37, Caltech, 1200 E.~California Blvd. Pasadena, CA 91125, USA}
\email{matilde@caltech.edu} 
\urladdr{http://www.its.caltech.edu/~matilde}

\address{Gon{\c c}alo Tabuada, Department of Mathematics, MIT, Cambridge, MA 02139, USA}
\email{tabuada@math.mit.edu}
\urladdr{http://math.mit.edu/~tabuada}
\thanks{Sections \ref{sec:nil-orbit}-\ref{sec:Proof1} are a revised version of half of an old preprint \cite{old}.
The remaining work is a recent collaboration between the first and the last author.}
\date{\today}


\begin{document}
\begin{abstract}
In this article we extend Voevodsky's nilpotence conjecture from smooth projective schemes to the broader setting of smooth proper dg categories.
Making use of this noncommutative generalization, we then address Voevodsky's original conjecture in the following cases: quadric fibrations, intersection of
quadrics, linear sections of Grassmannians, linear sections of determinantal varieties, homological projective duals, and Moishezon manifolds.

\end{abstract}

\maketitle
\vskip-\baselineskip
\vskip-\baselineskip
\vskip-\baselineskip
\section{Introduction and statement of results}
Let $k$ be a base field and $F$ a field of coefficients of characteristic zero.
\subsection*{Voevodsky's nilpotence conjecture}
In a foundational work \cite{Voevodsky}, Voevodsky introduced the smash-nilpotence equivalence relation $\sim_{\otimes\mathrm{nil}}$ on algebraic
cycles and conjectured its agreement with the classical numerical equivalence relation $\sim_{\mathrm{num}}$. Concretely, given a smooth projective
$k$-scheme $X$, he stated the following:

\smallskip

{\it Conjecture $V(X)$: $\cZ^\ast(X)_F/\!\!\sim_{\otimes\mathrm{nil}}=\cZ^\ast(X)_F/\!\!\sim_{\mathrm{num}}$.} 

\smallskip

Thanks to the work of Kahn-Sebastian, Matsusaka, Voevodsky, and Voisin (see \cite{KS,Matsusaka,Voevodsky,Voisin} and also \cite[\S11.5.2.3]{Andre}),
the above conjecture holds in the case of curves, surfaces, and abelian $3$-folds (when $k$ is of characteristic zero).
\subsection*{Noncommutative nilpotence conjecture}
A {\em dg category} $\cA$ is a category enriched over dg $k$-vector spaces; see \S\ref{sub:dg}. Following Kontsevich \cite{IAS,Miami,finMot},
$\cA$ is called {\em smooth} if it is perfect as a bimodule over itself and {\em proper} if for any two objects $x,y \in \cA$ we have
$\sum_i \mathrm{dim} \,H^i\cA(x,y) < \infty$. The classical example is the unique dg enhancement $\perf_\dg(X)$ of the category of perfect
complexes $\perf(X)$ of a smooth projective $k$-scheme $X$; see Lunts-Orlov \cite{LO}. As explained in \S\ref{sub:nilpotence}-\ref{sub:numerical}, 
the Grothendieck group $K_0(\cA)$ of every smooth proper dg category $\cA$ comes endowed with a $\otimes$-nilpotence equivalence relation 
$\sim_{\otimes \mathrm{nil}}$ and with a numerical equivalence relation $\sim_{\mathrm{num}}$. Motivated by the above conjecture, we state the~following:

\smallskip

{\it Conjecture $V_{NC}(\cA)$: $K_0(\cA)_F/\!\!\sim_{\otimes \mathrm{nil}}=K_0(\cA)_F/\!\!\sim_{\mathrm{num}}$.}

\smallskip

Our first main result is the following reformulation of Voevodsky's conjecture:
\begin{theorem}\label{thm:main1}
Conjecture $V(X)$ is equivalent to conjecture $V_{NC}(\perf_\dg(X))$.
\end{theorem}
Theorem~\ref{thm:main1} shows us that when restricted to the commutative world, the noncommutative nilpotence conjecture reduces to Voevodsky's
original conjecture. Making use of this noncommutative viewpoint, we now address Voevodsky's nilpotence conjecture in several cases.

\subsection*{Quadric fibrations}
Let $S$ be a smooth projective $k$-scheme and $q:Q \to S$ a flat quadric fibration of relative dimension $n$ with $Q$ smooth. Recall from Kuznetsov \cite{kuznetquadrics} (see also  \cite{auel-bernar-bologn}) the construction of the sheaf $\cC_0$ of even parts of the Clifford algebra associated to $q$. Recall also from {\em loc. cit.} that when the discriminant divisor of $q$ is smooth and $n$ is even  (resp. odd) we have a discriminant double cover $\widetilde{S} \to S$ (resp. a square root stack $\widehat{S}$)
equipped with an Azumaya algebra $\cB_0$. Our second main result allows us to decompose conjecture $V(Q)$ into simpler pieces:
\begin{theorem}\label{thm:main2}
The following holds:
\begin{itemize}
\item[(i)] We have $V(Q) \Leftrightarrow V_{NC}(\perf_\dg(S, \cC_0)) + V(S)$.
\item[(ii)] When the discriminant divisor of $q$ is smooth and $n$ is even, we have $V(Q) \Leftrightarrow V(\widetilde{S}) + V(S)$. As a consequence, $V(Q)$ holds when $\mathrm{dim}(S)\leq 2$,
and becomes equivalent to $V(\widetilde{S})$ when $S$ is an abelian $3$-fold and $k$ is of characteristic zero. 
\item[(iii)] When the discriminant divisor of $q$ is smooth and $n$ is odd, we have $V(Q) \Leftrightarrow V_{NC}(\perf_\dg(\widehat{S},\cB_0)) + V(S)$.
As a consequence, $V(Q)$ becomes equivalent to $V_{NC}(\perf_\dg(\widehat{S}), \cB_0)$ when $\dim(S) \leq 2$.
This latter conjecture holds when $\dim(S) \leq 2$. 
\end{itemize}
\end{theorem}
\begin{remark} 
The (rational) Chow motive of a quadric fibration $q: Q \to S$ was computed by Vial in \cite[Thm. 4.2 and  Cor. 4.4]{vial-fibrations}.
In the particular case where $\mathrm{dim}(S) \leq 2$, it consists of a direct sum of submotives of smooth projective $k$-schemes of dimension at most two.
This motivic decomposition provides an alternative ``geometric'' proof of conjecture $V(Q)$. We will rely on this argument to prove
the last statement of item (iii). In the particular case where $S$ is a curve and $k$ is algebraically closed, we provide also a ``categorical'' proof of this last statement; see Remark \ref{rk:new}. The fact that $V_{NC}(\perf_\dg(\widehat{S}), \cB_0)$ holds when $\dim(S) \leq 2$
will play a key role in the proof of Theorem \ref{thm:main3} below.
\end{remark}

\subsection*{Intersection of quadrics}
Let $X$ be a smooth complete intersection of $r$ quadric hypersurfaces in $\bbP^m$. The linear span of these $r$ quadrics gives rise to a hypersurface
$Q \subset \bbP^{r-1} \times \bbP^m$, and the projection into the first factor to a flat quadric fibration $q: Q \to \bbP^{r-1}$
of relative dimension $m-1$. 
\begin{theorem}\label{thm:main3}
The following holds:
\begin{itemize}
\item[(i)] We have $V(X) \Leftrightarrow V_{NC}(\perf_\dg(\bbP^{r-1}, \cC_0))$.
\item[(ii)] When the discriminant divisor of $q$ is smooth and $m$ is odd, we have $V(X) \Leftrightarrow V(\widetilde{\bbP^{r-1}})$.
As a consequence, $V(X)$ holds when $r\leq 3$.
\item[(iii)] When the discriminant divisor of $q$ is smooth  and $m$ is even, we have $V(X) \Leftrightarrow V_{NC}(\perf_\dg(\widehat{\bbP^{r-1}}, \cB_0))$. This latter conjecture holds when $r\leq 3$ and $k$ is algebraically closed.
\end{itemize}
\end{theorem} 
\begin{remark}
The (rational) Chow motive of a complete intersection $X$ was computed in \cite[Thm.~2.1]{bernardara-tabuada-paranj} in the particular cases where
$r\leq 2$ or $r=3$ and $m$ is odd. It consists of a direct sum of submotives of smooth projective $k$-schemes of dimension at most one. This
motivic decomposition provides an alternative proof of conjecture $V(X)$. A similar argument holds in the case where $r=3$ and $m$ is~even.
\end{remark}
\begin{remark}{(Relative version)}
Theorem~\ref{thm:main3} has a relative analogue with $X$ replaced by a generic relative complete intersection $X \to S$ of $r$ quadric fibrations $Q_i \to S$
of relative dimension $m-1$; consult \cite[Def. 1.2.4]{auel-bernar-bologn} for details. Items (i), (ii), and (iii), hold similarly with $\bbP^{r-1}$ replaced by a $\bbP^{r-1}$-bundle
$T \to S$, with $V(\widetilde{\bbP^{r-1}})$ replaced by $V(\widetilde{T}) + V(S)$, and with $V_{NC}(\perf_\dg(\widehat{\bbP^{r-1}}, \cB_0))$ replaced by $V_{NC}(\perf_\dg(\widehat{T}, \cB_0))
+ V(S)$, respectively. Note that thanks to the relative item (ii), conjecture $V(X)$ holds when $r=2$ and $S$ is a curve. 
\end{remark}
\subsection*{Linear sections of Grassmanians}
Following Kuznetsov \cite{kuznet-grass}, consider the following two classes of schemes:
\begin{itemize}
\item[(i)] Let $X_L$ be a generic linear section of codimension $r$ of the Grassmannian $\mathrm{Gr}(2,W)$ (with $W=k^{\oplus 6}$) under the Pl\"ucker embedding, and $Y_L$ the
 corresponding dual linear section of the cubic Pfaffian $\mathrm{Pf}(4,W^*)$ in $\bbP(\Lambda^2 W^*)$.
\end{itemize}
For example when $r=3$, $X_L$ is a Fano $5$-fold; when $r=4$, $X_L$ is a Fano $4$-fold; and when $r=6$, $X_L$ is a $K3$ surface of degree 14 and $Y_L$ a Pfaffian cubic $4$-fold. Moreover, $X_L$ and $Y_L$ are smooth whenever $r \leq 6$.
\begin{itemize}
\item[(ii)] Let $X_L$ be a generic linear section of codimension $r$ of the Grassmannian $\mathrm{Gr}(2,W)$ (with $W=k^{\oplus 7}$) under the Pl\"ucker embedding, and $Y_L$ the
 corresponding dual linear section of the cubic Pfaffian $\mathrm{Pf}(4,W^*)$ in $\bbP(\Lambda^2 W^*)$.
\end{itemize}
For example when $r=5$, $X_L$ is a Fano $5$-fold; when $r=4$, $X_L$ is a Fano $4$-fold; when $r=8$, $Y_L$ is a Fano $4$-fold; and when $r=9$, $Y_L$
is a Fano $5$-fold. Moreover, $X_L$ and $Y_L$ are smooth whenever $r \leq 10$.

\begin{theorem}\label{thm:main-grass}
Let $X_L$ and $Y_L$ be as in the above classes (i)-(ii). Under the assumption that $X_L$ and $Y_L$ are smooth, we have $V(X_L) \Leftrightarrow V(Y_L)$. This conjecture holds when $r\leq 6$ (class (i)), and when $r\leq 6$ and $8 \leq r \leq 10$ (class (ii)).
\end{theorem}
\begin{remark}
To the best of the authors' knowledge, Theorem \ref{thm:main-grass} proves Voevodsky's nilpotence conjecture in new cases.
\end{remark}
\subsection*{Linear sections of determinantal varieties}
Let $U$ and $V$ be $k$-vector spaces of dimensions $m$ and $n$, respectively, with $n \geq m$, and  $0 < r < m$ an integer.
Following \cite{berna-bolo-faenzi},
consider the universal determinantal variety $Z_{m,n}^r \subset \bbP(U \otimes V)$ given by the locus
of matrices $M: U \to V^*$ 
of rank at most $r$. Its Springer resolution is denoted by $\cX_{m,n}^r:= \bbP(\cQ \otimes U) \to \mathrm{Gr}(r,U)$, where $\cQ$ stands for
the tautological quotient on $\mathrm{Gr}(r,U)$. Under these notations, we have the following class of schemes:
\begin{itemize}
\item[(i)] Let $X_L$ be a generic linear section of codimension $c$ of $\cX_{m,n}^r$ under the map $\cX_{m,n}^r \to \bbP(U \otimes V)$, and $Y_L$ the
 corresponding dual linear section of $\cX_{m,n}^{m-r}$ under the map $\cX_{m,n}^{m-r} \to \bbP(U^* \otimes V^*)$.
\end{itemize}
\begin{remark}
As explained in \cite[\S3]{berna-bolo-faenzi}, $X_L$ and $Y_L$ are smooth crepant categorical resolution of singularities of $(Z_{m,n}^r)_L$ and
$(Z_{m,n}^{m-r})_L$, respectively.
\end{remark}
For example when $m=n=4$ and $r=1$, $X_L$ is a $(6-c)$-dimensional section of $\bbP^3 \times \bbP^3$ under the Segre
embedding, and $Y_L$ the resolution of the dual $(c-2)$-dimensional determinantal quartic.
In the same vein, when $m=n=4$ and $r=2$, $X_L$ is a $(11-c)$-dimensional section of the self-dual orbit of $4 \times 4$ matrices of rank $2$,
and $Y_L$ the $(c-5)$-dimensional dual section. Moreover, the following holds:
\begin{itemize}
\item[(a)] When $c \leq 7$, $X_L$ is a Fano $(11-4)$-fold and $\dim(Y_L) \leq 2$.
\item[(b)] When $c=8$, $X_L$ and $Y_L$ are dual Calabi-Yau threefolds
\item[(c)] When $c\geq 9$, $Y_L$ is a Fano ($c-5$)-fold and $\dim(X_L) \leq 2$.
\end{itemize}
For further example, consult \cite[\S 3.3]{berna-bolo-faenzi} and well as Tables $1$ and $2$ in {\em loc. cit.}
\begin{theorem}\label{thm:main-deter}
When $X_L$ and $Y_L$ are as in the above class (i), we have $V(X_L) \Leftrightarrow V(Y_L)$.
This conjecture holds when $\dim(X_L) \leq 2$ or $\dim(Y_L) \leq 2$.
\end{theorem}
\begin{remark}
To the best of the authors' knowledge, Theorem \ref{thm:main-deter} proves Voevodsky's nilpotence conjecture in new cases.
\end{remark}
\subsection*{Homological projective duality}
Making use of Kuznetsov's theory of homological projective duality (HPD) \cite{kuznetsov:hpd}, Theorem \ref{thm:main-grass} admits the following generalization:
let $X$ be a smooth projective $k$-scheme equipped with an ample line bundle $\cO_X(1)$. Note that $\cO_X(1)$ gives rise to a morphism $X \to \bbP(V)$,
where $V:=H^0(X,\cO_X(1))^*$. Let $Y$ be the
HP-dual of $X$, $\cO_Y(1)$ the associated ample line bundle, and $Y \to \bbP(V^\ast)$ the associated morphism. Assume that $\perf(X)$ admits a
{\em Lefschetz decomposition $\perf(X)=\langle \bbA_0, \ldots, \bbA_n(n)\rangle$ with respect to $\cO_X(1)$}, \ie a semi-orthogonal decomposition of
$\perf(X)$ such that $\bbA_0 \supset \ldots \supset \bbA_n$ and $\bbA_i(i):= \bbA_i\otimes \cO(i)$. Assume also that conjecture $V_{NC}(\bbA_0^\dg)$ holds, where
$\bbA_0^\dg$ stands for the dg enhancement of $\bbA_0$ induced from $\perf_\dg(X)$; see \S\ref{sub:perfect}. Finally, let $L \subset V$ be a subspace such that the
linear sections $X_L:= X \times_{\bbP(V)}\bbP(L^{\perp})$ and $Y_L:= Y \times_{\bbP(Y^\ast)}\bbP(L)$ are of expected dimension
$\mathrm{dim}(X_L)=\mathrm{dim}(X) - \mathrm{dim}(L)$ and $\mathrm{dim}(Y_L)= \mathrm{dim}(Y) - \mathrm{dim}(L^\perp)$, respectively.

\begin{theorem}\label{thm:main-hpd}
Let $X_L$ and $Y_L$ be as above. Under the assumption that $X_L$ and $Y_L$ are smooth, we have $V(X_L) \Leftrightarrow V(Y_L)$.
\end{theorem}
\begin{remark}
Theorem \ref{thm:main-hpd} reduces to Theorem \ref{thm:main-grass} (resp. to Theorem \ref{thm:main-deter}) in the particular case of Grassmanian-Pfaffian
(resp. determinantal) homological projective duality; consult \cite{kuznet-grass} (resp. \cite{berna-bolo-faenzi}) for details.
\end{remark}
\subsection*{Moishezon manifolds}
A {\em Moishezon manifold} $X$ is a compact complex manifold such that the field of meromorphic functions on each component of $X$ has transcendence degree equal to the dimension of the component. As proved by Moishezon \cite{Moishezon}, $X$ is a smooth projective $\bbC$-scheme if and only if it admits a K\"ahler metric. In the remaining cases, Artin \cite{Artin-Moish} showed that $X$ is a proper algebraic space over $\bbC$.

Let $Y \to \bbP^2$ be one of the non-rational conic bundles described by Artin and Mumford in \cite{artin-mumford}, and $X \to Y$ a small
resolution. In this case, $X$ is a smooth (non necessarily projective) Moishezon manifold.

\begin{theorem}\label{thm:main-moish-ex}
Conjecture $V_{NC}(\perf_\dg(X))$ holds for the above resolutions.
\end{theorem}

\begin{remark}
The proofs of Theorems \ref{thm:main2}, \ref{thm:main3}, \ref{thm:main-grass}, \ref{thm:main-deter}, \ref{thm:main-hpd}, and \ref{thm:main-moish-ex} are based on the study of a smooth projective
$k$-scheme (or algebraic space) $X$ via semi-orthogonal decompositions of its category of perfect complexes $\perf(X)$; see Bondal-Orlov \cite{BO}
and Kuznetsov \cite{kuz:ICM2014} for instance. This approach allows the reduction of Voevodsky's conjecture $V(X)$ to several noncommutative conjectures
$V_{NC}$ - one for each piece of the semi-orthogonal decomposition. We believe this provides a new tool for the proof of Voevodsky's conjecture as well as of 
its generalization to algebraic spaces. 
\end{remark}

\section{Preliminaries}
\subsection{Dg categories}\label{sub:dg}
A {\em differential graded (=dg) category} $\cA$ is a category enriched over dg $k$-vector spaces; consult Keller \cite{ICM} for details. For example, 
every (dg) $k$-algebra $A$ gives naturally rise to a dg category $\underline{A}$ with a single object. Let $\dgcat$ be the category 
of small dg categories. Recall from \cite[\S3]{ICM} the construction of the derived category $\cD(\cA)$
of $\cA$. This triangulated category admits arbitrary
direct sums and we will write $\cD_c(A)$ for the full subcategory of compact objects. A dg functor $\cA \to \cB$ is called
a {\em 
Morita equivalence} if it induces an equivalence $\cD(\cA) \stackrel{\sim}{\to} \cD(\cB)$. Finally, let us write $\cA\otimes \cB$ for the tensor product  of dg categories.
\subsection{Perfect complexes}\label{sub:perfect}
Given a stack $\cX$ and a sheaf of $\cO_{\cX}$-algebras $\cG$, let $\mathrm{Mod}(\cX,\cG)$ be the Grothendieck category of sheaves of (right) $\cG$-modules, $\cD(\cX,\cG):= \cD(\mathrm{Mod}(\cX,\cG))$ the derived category of $\cG$, and $\perf(\cX,\cG)$ the subcategory of perfect complexes. As explained in \cite[\S4.4]{ICM}, the derived category $\cD_\dg(\cE x)$ of an abelian (or exact) category $\cE x$ is defined as the (Drinfeld's) dg quotient $\cC_\dg(\cE x)/\mathrm{Ac}_\dg(\cE x)$ of the dg category of complexes over $\cE x$ by its full dg subcategory of acyclic complexes. Hence, let us write $\cD_\dg(\cX,\cG)$ for the dg category $\cD_\dg(\cE x )$ with $\cE x:= \mathrm{Mod}(\cX,\cG)$ and $\perf_\dg(\cX,\cG)$ for the full dg subcategory of perfect complexes.
\begin{lemma}\label{lem:smoothness}
Let $X$ be a smooth projective $k$-scheme and $\perf(X) =\langle \cT_1, \ldots, \cT_n\rangle$ a semi-orthogonal decomposition. In this case, the dg categories
$\cT_i^\dg$ (where $\cT_i^\dg$ stands for the dg enhancement of $\cT_i$ induced from $\perf_\dg(X)$) are smooth and proper. 
\end{lemma}
\begin{proof}
Let $\Ho(\dgcat)$ be the localization of $\dgcat$ with respect to the class of Morita equivalences. The tensor product of dg categories gives rise to a symmetric monoidal 
structure on $\dgcat$ which descends to $\Ho(\dgcat)$. Moreover, as proved
in \cite[Thm.~5.8]{CT1}, the smooth and proper dg categories can be characterized as those objects of $\Ho(\dgcat)$ which are dualizable. Note that the canonical inclusion
$\cT^\dg_i \hookrightarrow \perf_\dg(X)$ and projection $\perf_\dg(X) \to \cT^\dg_i$ dg functors express $\cT_i^\dg$ as a direct factor of $\perf_\dg(X)$ in $\Ho(\dgcat)$.
Hence, since $\perf_\dg(X)$ is smooth and proper, we  conclude that $\cT_i^\dg$ is also smooth and proper.
\end{proof}
\subsection{$\otimes$-nilpotence equivalence relation}\label{sub:nilpotence}
Let $\cA$ be a dg category. An element $[M]$ of the Grothendieck group $K_0(\cA):=K_0(\cD_c(\cA))$ is called {\em $\otimes$-nilpotent} if there exists
an integer $n >0$ such that $[M^{\otimes n}]=0$ in $K_0(\cA^{\otimes n})$. This gives rise to a well-defined equivalence relation
$\sim_{\otimes \mathrm{nil}}$ on $K_0(\cA)$ and on its $F$-linearization $K_0(\cA)_F$.
\subsection{Numerical equivalence relation}\label{sub:numerical}
Let $\cA$ be a smooth proper dg category. As explained in \cite[\S4]{Kontsevich}, the pairing $(M,N) \mapsto \sum_i (-1)^i \mathrm{dim}\, \Hom_{\cD_c(\cA)}(M,N[i])$ gives rise to a well-defined bilinear form $\chi(-,-)$ on $K_0(\cA)$. Moreover, the left and right kernels of $\chi(-,-)$ are the same. An element
$[M]$ of the Grothendieck group $K_0(\cA)$ is called {\em numerically trivial} if $\chi([M],[N])=0$ for all $[N] \in K_0(\cA)$. This gives rise 
to an equivalence relation $\sim_{\mathrm{num}}$ on $K_0(\cA)$ and consequently on $K_0(\cA)_F$. When $\cA=\perf_\dg(X)$, with $X$ a smooth projective $k$-scheme,
and $F=\bbQ$ this equivalence relation reduces, via the Chern character $K_0(X)_\bbQ \stackrel{\sim}{\to} CH^\ast(X)_\bbQ$, to the classical numerical equivalence relation
on the Chow ring $CH^\ast(X)_\bbQ$.
\subsection{Motives}\label{sub:motives}
We assume the reader is familiar with the categories of Chow motives $\Chow(k)_F$ and numerical motives $\Num(k)_F$; see \cite[\S4]{Andre}. The 
Tate motive will be denoted $F(1)$. In the same vein, we assume some familiarity with the categories of noncommutative Chow motives $\NChow(k)_F$ 
and noncommutative numerical motives $\NNum(k)_F$; consult the surveys \cite[\S 2-3]{survey} \cite[\S4]{survey-1} and the references therein.
Recall from {\em loc.} {\em cit.} that $\NNum(k)_F$ is the idempotent completion of the quotient of $\NChow(k)_F$ by its largest $\otimes$-ideal\footnote{Different from the entire category $\NChow(k)_F$.}, and that $\Hom_{\NChow(k)_F}(\underline{k},\cA)\simeq K_0(\cA)_F$.
\section{Orbit categories and $\otimes$-nilpotence}\label{sec:nil-orbit}
Let $\cC$ be an $F$-linear additive rigid symmetric monoidal category.
\subsection*{Orbit categories}
Given a $\otimes$-invertible object $\cO \in \cC$, recall from \cite[\S7]{CvsNC} the construction of the orbit category $\cC\!/_{\!\!-\otimes \cO}$.
It has the same objects as $\cC$ and morphisms
$$ \Hom_{\cC\!/_{\!\!-\otimes \cO}}(a,b) := \oplus_{j \in \bbZ} \Hom_{\cC}(a,b \otimes \cO^{\otimes j})\,.$$   
The composition law is induced from $\cC$. By construction, $\cC\!/_{\!\!-\otimes \cO}$ is $F$-linear, additive, and comes equipped with a 
canonical projection functor $\pi: \cC \to \cC\!/_{\!\!-\otimes \cO}$. Moreover, $\pi$ is endowed with a natural $2$-isomorphism 
$\pi \circ (-\otimes \cO) \stackrel{\sim}{\Rightarrow}\pi$ and is $2$-universal among all such functors.  As proved in \cite[Lem.~7.3]{CvsNC},
$\cC\!/_{\!\!-\otimes \cO}$ inherits from $\cC$ a symmetric monoidal structure making $\pi$ symmetric monoidal. On objects it is the same. 
On morphisms it is defined as the unique bilinear pairing
$$ \underset{j \in \bbZ}{\oplus} \Hom_\cC(a,b\otimes \cO^{\otimes j}) \times \underset{j \in \bbZ}{\oplus} \Hom_\cC(c,d\otimes \cO^{\otimes j})
\too \underset{j \in \bbZ}{\oplus} \Hom_\cC(a\otimes c,(b\otimes d) \otimes \cO^{\otimes j})$$
which sends the pair ($a \stackrel{f_r}{\to} b \otimes \cO^{\otimes r}, c \stackrel{g_s}{\to} d \otimes \cO^{\otimes s})$ to
$$ (f\otimes g)_{(r+s)}: a\otimes c \stackrel{f_r \otimes g_s}{\too} b \otimes \cO^{\otimes r} \otimes d \otimes \cO^{\otimes s} \simeq
(b\otimes d) \otimes \cO^{\otimes(r +s)}\,.$$
\subsection*{$\otimes$-nilpotence}
The $\onil$-ideal of $\cC$ is defined as 
$$ \onil(a,b):= \{ f \in \Hom_\cC(a,b) \, |\, f^{\otimes n}=0 \,\,\mathrm{for} \,\, n \gg 0\}\,.$$
By construction, $\onil$ is a $\otimes$-ideal. Moreover, all its ideals $\onil(a,a)\subset \Hom_\cC(a,a)$ are nilpotent; see
\cite[Lem.~7.4.2 (ii)]{AK}. As a consequence, the $\otimes$-functor $\cC \to \cC\!/\!\onil$ is not only $F$-linear and additive but 
moreover conservative. Furthermore, since idempotents can be lifted along nilpotent ideals (see \cite[\S III Prop.~2.10]{Bass}), $\cC\!/\!\onil$ 
is idempotent complete whenever $\cC$ is idempotent complete.
\subsection*{Compatibility}
Let $\cC$ be a category and $\cO \in \cC$ a $\otimes$-invertible objects as above.
\begin{proposition}\label{prop:orbit}
There exists a canonical $F$-linear additive $\otimes$-equivalence $\theta$ making the following diagram commute:
\begin{equation}\label{eq:diagram-key}
\xymatrix@C=2em@R=2em{
\cC\!/\!\onil \ar[d] & \cC \ar[l] \ar[r] & \cC\!/_{\!\!-\otimes \cO} \ar[d] \\
(\cC\!/\!\onil)\!/_{\!\!-\otimes \cO} \ar[rr]^-\simeq_-\theta && (\cC\!/_{\!\!-\otimes \cO})\!/\!\!\onil\,.
}
\end{equation}
\end{proposition}
\begin{proof}
The existence of the $F$-linear additive $\otimes$-functor $\theta$ follows from the fact that
\begin{equation}\label{eq:composed}
\cC \too \cC\!/_{\!\!-\otimes \cO} \too (\cC\!/_{\!\!-\otimes \cO})\!/\!\onil
\end{equation}
vanishes on the $\onil$-ideal and also from the natural $2$-isomorphism between $\eqref{eq:composed} \circ (-\otimes \cO)$ and $\eqref{eq:composed}$. Note that the functor $\theta$ is the identity on objects and sends $\{[f_j]\}_{j \in \bbZ}$ to $[\{f_j\}_{j \in \bbZ}]$. Clearly, it is full. The faithfulness is left as an exercise.
\end{proof}
\section{$\otimes$-nilpotence of motives}
By construction, the categories $\Chow(k)_F$ and $\NChow(k)_F$ are $F$-linear, additive, rigid symmetric monoidal, and idempotent complete. 
Let us denote by 
\begin{eqnarray*}
\Voev(k)_F:= \Chow(k)_F\!/\onil & \mathrm{and} & \NVoev(k)_F := \NChow(k)_F\!/\onil
\end{eqnarray*}
the associated quotients. They fit in the following sequences
\begin{eqnarray*}
\Chow(k)_F \to \Voev(k)_F \to \Num(k)_F & \NChow(k)_F \to \NVoev(k)_F \to \NNum(k)_F\,.&
\end{eqnarray*}
The relation between all these motivic categories is the following:
\begin{proposition}\label{prop:compatibility}
There exist $F$-linear additive fully-faithful $\otimes$-functors $R, R_{\onil}, R_\cN$ making the following diagram commute:
\begin{equation}\label{diag:main}
\xymatrix@C=2em@R=2em{
\Chow(k)_F \ar[d] \ar[r]^-\pi &  \Chow(k)_F\!/_{\!\!-\otimes F(1)}  \ar[d]  \ar[r]^-R&  \NChow(k)_F\ar[d]  \\
\Voev(k)_F \ar[d]  \ar[r]^-\pi& \Voev(k)_F\!/_{\!\!-\otimes F(1)} \ar[d] \ar[r]^-{R_{\onil}}&  \NVoev(k)_F \ar[d]\\
\Num(k)_F  \ar[r]^-\pi& \Num(k)_F\!/_{\!\!-\otimes F(1)} \ar[r]_-{R_\cN} & \NNum(k)_F\,. \\
}
\end{equation}
\end{proposition}
\begin{proof}
The outer commutative square, with $R, R_\cN$ $F$-linear additive fully-faithful $\otimes$-functors, was built in \cite[Thm.~1.13]{AMJ}. Consider now the ``zoomed'' diagram:
\begin{equation}\label{eq:diagram-aux}
\xymatrix@C=2em@R=2em{
\Chow(k)_F\!/_{\!\!-\otimes F(1)} \ar[d]  \ar@{=}[r] & \Chow(k)_F\!/_{\!\!-\otimes F(1)} \ar[r]^-R\ar[d]&  \NChow(k)_F\ar[d]  \\
\Voev(k)_F\!/_{\!\!-\otimes F(1)} \ar[d] \ar[r]^-\simeq_-\theta&(\Chow(k)_F\!/_{\!\!-\otimes F(1)})\!/\!\onil \ar[r]^-{R/\!\onil} \ar[d]&  \NVoev(k)_F \ar[d]\\
\Num(k)_F\!/_{\!\!-\otimes F(1)} \ar@{=}[r] &\Num(k)_F\!/_{\!\!-\otimes F(1)} \ar[r]_-{R_\cN} & \NNum(k)_F\,. \\
}
\end{equation}
By definition, $R_{\onil}:=R/\!\onil \circ\, \theta$. Hence, since $R$ is an $F$-linear additive fully-faithful $\otimes$-functor, we conclude that $R_{\onil}$ is also an $F$-linear additive fully-faithful $\otimes$-functor. The 
commutativity of the bottom squares of diagram \eqref{eq:diagram-aux} follows from the fact that $\Num(k)_F\!/_{\!\!-\otimes F(1)}$ 
identifies with the quotient of $\Chow(k)_F\!/_{\!\!-\otimes F(1)}$ by its largest $\otimes$-ideal $\cN$; consult \cite[Prop.~3.2]{AMJ} for details. 
\end{proof}
\section{Proof of Theorem~\ref{thm:main1}}\label{sec:Proof1}
Note first that we have the following natural isomorphisms
\begin{eqnarray*}
\Hom_{\Voev(k)_F\!/_{\!\!-\otimes F(1)}}(\mathrm{Spec}(k),X) & \simeq & \cZ^\ast(X)_F/\!\!\sim_{\otimes\mathrm{nil}} \\
\Hom_{\Num(k)_F\!/_{\!\!-\otimes F(1)}}(\mathrm{Spec}(k),X) & \simeq & \cZ^\ast(X)_F/\!\!\sim_{\mathrm{num}}\,.
\end{eqnarray*}
As a consequence, conjecture $V(X)$ becomes equivalent to the injectivity of
\begin{equation}\label{eq:induced-1}
\Hom_{\Voev(k)_F\!/_{\!\!-\otimes F(1)}}(\mathrm{Spec}(k),X) \twoheadrightarrow \Hom_{\Num(k)_F\!/_{\!\!-\otimes F(1)}}(\mathrm{Spec}(k),X)\,.
\end{equation}
Given a smooth and proper dg category $\cA$, we have also natural isomorphisms
\begin{eqnarray*}
\Hom_{\NVoev(k)_F}(\underline{k},\cA)\simeq K_0(\cA)_F/\!\!\sim_{\otimes \mathrm{nil}} &&
\Hom_{\NNum(k)_F}(\underline{k},\cA) \simeq K_0(\cA)_F/\!\!\sim_{\mathrm{num}}\,.
\end{eqnarray*} 
Hence, conjecture $V_{NC}(\cA)$ becomes equivalent to the injectivity of
\begin{equation}\label{eq:induced-2}
\Hom_{\NVoev(k)_F}(\underline{k},\cA) \twoheadrightarrow \Hom_{\NNum(k)_F}(\underline{k},\cA)\,.
\end{equation}
Now, recall from \cite[Thm.~1.1]{CvsNC} that the image of $X$ under the composed functor $R \circ \pi$ identifies naturally with the noncommutative 
Chow motive $\perf_\dg(X)$. Similarly, the image of $\mathrm{Spec}(k)$ under $R \circ \pi$ identifies with $\perf_\dg(\mathrm{Spec}(k))$ which is 
Morita equivalent to $\underline{k}$. As a consequence, since the functors $R_{\otimes_\mathrm{nil}}$ and $R_\cN$ are fully-faithful, the bottom 
right-hand side square of diagram \eqref{diag:main} gives rise to the following commutative diagram:
$$
\xymatrix{
\Hom_{\Voev(k)_F\!/_{\!\!-\otimes F(1)}}(\mathrm{Spec}(k), X) \ar@{->>}[d]_-{\eqref{eq:induced-1}} \ar[r]^-\simeq & \Hom_{\NVoev(k)_F}(\underline{k}, \perf_\dg(X)) \ar@{->>}[d]^-{\eqref{eq:induced-2}} \\
\Hom_{\Num(k)_F\!/_{\!\!-\otimes F(1)}} (\mathrm{Spec}(k),X) \ar[r]_-\simeq & \Hom_{\NNum(k)_F}(\underline{k}, \perf_\dg(X))\,.
}
$$
Using the above reformulations of conjectures $V$ and $V_{NC}$, we conclude finally that conjecture $V(X)$ is equivalent to 
conjecture $V_{NC}(\perf_\dg(X))$.
\section{Proof of Theorem~\ref{thm:main2}}
{\bf Item (i)}. As proved by Kuznetsov in \cite[Thm.~4.2]{kuznetquadrics}, one has the following semi-orthogonal decomposition 
$$ \perf(Q)=\langle\perf(S,\cC_0), \perf(S)_1, \ldots, \perf(S)_n \rangle$$
with $\perf(S)_i:= q^\ast \perf(S) \otimes \cO_{Q/S}(i)$. Note that $\perf(S)_i \simeq \perf(S)$ for every $i$. Using \cite[Prop.~3.1]{bernardara-tabuada-paranj}, one then obtains a direct sum decomposition in $\NChow(k)_F$
\begin{equation}\label{eq:decomp}
\perf_\dg(Q)\simeq \perf^\dg(S,\cC_0) \oplus \underbrace{\perf_\dg(S) \oplus \cdots \oplus \perf_\dg(S)}_{n\,\,\text{copies}}\,,
\end{equation}
where $\perf^\dg(S,\cC_0)$ stands for the dg enhancement of $\perf(S,\cC_0)$ induced from $\perf_\dg(Q)$. Note that thanks to Lemma~\ref{lem:smoothness}, the dg category $\perf^\dg(S,\cC_0)$ is smooth and proper. Since the inclusion of categories $\perf(S,\cC_0) \hookrightarrow \perf(Q)$ is of Fourier-Mukai type
(see \cite[Prop. 4.9]{kuznetquadrics}), its kernel $\cK \in \perf(S \times Q, \cC_0^\op \boxtimes \cO_X)$ gives rise to a Fourier-Mukai
Morita equivalence $\Phi^\cK_\dg: \perf_\dg(S,\cC_0) \to \perf^\dg(S,\cC_0)$. Hence, we can replace in the above decomposition \eqref{eq:decomp} the dg category $\perf^\dg(S,\cC_0)$ by the canonical one $\perf_\dg(S,\cC_0)$ (see \S\ref{sub:perfect}). Finally, using the above description \eqref{eq:induced-2} of the noncommutative nilpotence conjecture, one concludes that 
\begin{equation}\label{eq:equivalences}
V_{NC}(\perf_\dg(Q)) \Leftrightarrow V_{NC}(\perf_\dg(S,\cC_0)) + V_{NC}(\perf_\dg(S))\,.
\end{equation}
The proof follows now automatically from Theorem~\ref{thm:main1}.

{\bf Item (ii)}. As proved by Kuznetsov in \cite[Prop. 3.13]{kuznetquadrics}, $\perf(S,\cC_0)$ is Fourier-Mukai equivalent
to $\perf(\widetilde{S},\cB_0)$. Hence, the above equivalence \eqref{eq:equivalences} reduces to
\begin{equation}\label{eq:equivalences-1}
V_{NC}(\perf_\dg(Q)) \Leftrightarrow V_{NC}(\perf_\dg(\widetilde{S},\cB_0)) + V_{NC}(\perf_\dg(S))\,.
\end{equation}
Since $\cB_0$ is a sheaf of Azumaya algebras and $F$ is of characteristic zero, the canonical dg functor $\perf_\dg(\widetilde{S}) \to \perf_\dg(\widetilde{S}, \cB_0)$ becomes an isomorphism in $\NChow(k)_F$; see \cite[Thm.~2.1]{VdB}. Consequently, conjecture $V_{NC}(\perf_\dg(\widetilde{S},\cB_0))$ reduces to conjecture $V_{NC}(\perf_\dg(\widetilde{S}))$. The proof follows now from Theorem~\ref{thm:main1}.

{\bf Item (iii)}. As proved by Kuznetsov in \cite[Prop. 3.15]{kuznetquadrics}, $\perf(S,\cC_0)$ is Fourier-Mukai equivalent
to $\perf(\widehat{S},\cB_0)$. 
Hence, the above equivalence \eqref{eq:equivalences} reduces to
\begin{equation}\label{eq:equivalences-2}
V_{NC}(\perf_\dg(Q)) \Leftrightarrow V_{NC}(\perf_\dg(\widehat{S},\cB_0)) + V_{NC}(\perf_\dg(S))\,.
\end{equation}
The proof of the first claim follows now from Theorem~\ref{thm:main1}. 

Let us now prove the second claim, which via \eqref{eq:equivalences-2} is equivalent to the proof of $V_{NC}(\perf_\dg(Q))$.
Thanks to Vial \cite[Thm. 4.2 and Cor. 4.4]{vial-fibrations}, the rational Chow motive $M_\bbQ(Q)$ of $Q$ decomposes as $M_\bbQ(Q)
= M_\bbQ(S)^{\oplus (n- \dim(S))} \oplus N,$ where $N$ stands for a submotive of a smooth projective $k$-scheme of dimension 
$\leq \dim(S)$. Therefore, when $\dim(S) \leq 2$, conjecture $V(Q)=V_{NC}(\perf_\dg(Q))$ holds.
\begin{remark}\label{rk:new}
Assume that $S$ is a smooth projective curve and that $k$ is algebraically closed. In this remark we provide a ``categorical'' proof of the second
claim of item (iii) of Theorem \ref{thm:main2}. Thanks to the work of Graber-Harris-Starr \cite{graber_harris_starr},
the fibration $q:Q \to S$ admits a section. Making use of it, we can perform reduction by hyperbolic splitting in order to obtain a conic bundle $q':Q' \to S$;
consult \cite[\S 1.3]{auel-bernar-bologn} for details. The sheaf $\cC_0'$ of even parts of the associated Clifford algebra is such that the
categories $\perf(S,\cC_0)$ and $\perf(S,\cC_0')$ are Fourier-Mukai equivalent; see \cite[Rk. 1.8.9]{auel-bernar-bologn}.
As a consequence, using \eqref{eq:equivalences} and the fact that $\mathrm{dim}(S)=1$, we obtain the following equivalence:
\begin{equation}\label{eq:conj-final}
V_{NC}(\perf_\dg(Q)) \Leftrightarrow V_{NC}(\perf_\dg(Q'))\,.
\end{equation}
Since $S$ is a curve, $Q'$ is a surface. Therefore, conjecture \eqref{eq:conj-final} holds.
\end{remark}
\section{Proof of Theorem~\ref{thm:main3}}
{\bf Item (i)}. As proved by Kuznetsov in \cite[Thm. 5.5]{kuznetquadrics}, we have a Fourier-Mukai equivalence
$\perf(X) \simeq \perf(\bbP^{r-1},\cC_0)$ when $m-2r+1 =0$, the following semi-orthogonal decomposition
\begin{equation*}
\perf(X)=\langle \perf(\bbP^{r-1}, \cC_0), \cO_X(1), \ldots, \cO_X(m-2r+1) \rangle\,,
\end{equation*}
when $m-2r+1 > 0$, and a dual semi-orthogonal decomposition of $\perf(\bbP^{r-1}, \cC_0)$ (containing a copy of $\perf(X)$ and exceptional objects)
when $m-2r+1 < 0$. The proof of the case $m-2r+1=0$ is clear. Let us now prove the case $m-2r+1>0$; the proof of the case $m-2r+1<0$ is similar. Using
\cite[Prop.~3.11]{bernardara-tabuada-paranj}, one obtains the following direct sum decomposition in $\NChow(k)_F$
$$ \perf_\dg(X) \simeq \perf^\dg(\bbP^{r-1},\cC_0) \oplus \underbrace{\perf_\dg(\underline{k}) \oplus \cdots \oplus \perf_\dg(\underline{k})}_{(m-2r +1) \,\,\text{copies}} \,.$$
Thanks to Lemma~\ref{lem:smoothness}, the dg category $\perf^\dg(\bbP^{r-1},\cC_0)$ is smooth and proper\footnote{In the case $m-2r+1<0$, these properties follow from the existence of a fully faithful Fourier-Mukai functor $\perf(\bbP^{r-1},\cC_0)\to \perf(Q)$, with $Q \subset \bbP^{r-1} \times \bbP^m$ a smooth hypersurface.}. 
Since 
$\perf(\bbP^{r-1},\cC_0) \hookrightarrow \perf(X)$ is of Fourier-Mukai type (see \cite[Prop. 4.9]{kuznetquadrics}), an argument similar to the one of the proof of Theorem~\ref{thm:main2}(i) shows us that 
$$ V_{NC}(\perf_\dg(X)) \Leftrightarrow V_{NC}(\perf_\dg(\bbP^{r-1},\cC_0)) + V_{NC}(\perf_\dg(\underline{k}))\,.$$
The proof follows now automatically from Theorem~\ref{thm:main1}.

{\bf Item (ii)-(iii)}. The proofs are similar to those of items (ii)-(iii) of Theorem \ref{thm:main2}. 
\section{Proof of Theorems~\ref{thm:main-grass} and \ref{thm:main-deter}}
Assume that $X_L$ and $Y_L$ are as in classes of (i)-(ii) of Theorem \ref{thm:main-grass} (resp. as in class (i) of Theorem \ref{thm:main-deter}).
As proved by Kuznetsov in \cite[\S10-11]{kuznet-grass} (resp. in \cite[Thm. 3.4]{berna-bolo-faenzi}), one of the following three situations occurs:
\begin{itemize}
\item[(a)] there is a semi-orthogonal decomposition $\perf(X_L) =\langle \perf(Y_L), \cE_1, \ldots, \cE_n \rangle$, with the $\cE_i$'s exceptional bundles on $X_L$;
\item[(b)] there is a semi-orthogonal decomposition $\perf(Y_L) =\langle \perf(X_L), \cE'_1, \ldots, \cE'_n \rangle$, with the $\cE'_i$'s exceptional bundles on $Y_L$;
\item[(c)] there is a Fourier-Mukai equivalence between $\perf(X_L)$ and $\perf(Y_L)$.
\end{itemize}
Therefore, equivalence $V(X_L) \Leftrightarrow V(Y_L)$ is clear in situation (c). Since the inclusions of categories
$\perf(Y_L) \hookrightarrow \perf(X_L)$
(situation (a)) and $\perf(X_L) \hookrightarrow \perf(Y_L)$ (situation (b)) are of Fourier-Mukai type,
a proof similar to the one of Theorem~\ref{thm:main2}(i)
shows us that equivalence $V(X_L)  \Leftrightarrow V(Y_L)$ also holds in situations (a)-(b). 
Note that this concludes the proof of Theorem \ref{thm:main-deter} since conjecture $V(X_L)$ (resp. $V(Y_L)$) holds when $\dim(X_L) \leq 2$ (resp. $\dim(Y_L) \leq 2$).

Let us now focus on class (i) of Theorem \ref{thm:main-grass}. The smooth projective $k$-schemes $X_L$ and $Y_L$ are of dimensions $8-r$ and $r-2$, respectively. Hence,
$V(Y_L)$ holds when $r\leq 4$ and $V(X_L)$ when $r=6$. When $r=5$, $X_L$ (and  $Y_L$) is a Fano $3$-fold. As explained by Gorchinskiy and Guletskii
in \cite[\S5]{gorch-gul-motives-and-repr}, the Chow motive of $X$ admits a decomposition into Lefschetz motives and submotives of
curves. This implies that $V(X_L)$ also holds.

Let us now focus on class (ii) of Theorem \ref{thm:main-grass}. The smooth projective $k$-schemes $X_L$ and $Y_L$ are of dimensions $10-r$ and $r-4$, respectively. 
Hence, $V(Y_L)$ holds when $r \leq 6$ and $V(X_L)$ when $r \geq 8$. This achieves the proof.
\section{Proof of Theorem~\ref{thm:main-hpd}}
Following Kuznetsov \cite[\S 4]{kuznetsov:hpd}, let us denote by $\mathfrak{a}_i$ the orthogonal complement of $\bbA_{i+1}$ in $\bbA_i$; 
these are called the ``primitive subcategories'' in {\em loc.} {\em cit.} Since the conjecture $V_{NC}(\bbA_0^\dg)$ holds, we hence 
conclude by induction that the conjectures $V_{NC}(\bbA_i^\dg)$ and $V_{NC}(\mathfrak{a}_i^\dg)$ also hold for every $i$. Thanks to 
HPD (see \cite[Thm. 6.3]{kuznetsov:hpd}), $\perf(Y)$ admits a Lefschetz decomposition $\perf(Y) = \langle \bbB_{m}(-m), \ldots, \bbB_0 \rangle$
with respect to $\cO_Y(1)$ such that the primitive subcategories $\mathfrak{b}_i$ coincide (via a Fourier-Mukai functor) with the primitive 
subcategories $\mathfrak{a}_i$. Consequently, $V_{NC}(\mathfrak{b}_i^\dg)$ holds for every $i$. An inductive argument, starting with
$\mathfrak{b}_m = \bbB_m$, allows us then to conclude that conjecture $V_{NC}(\bbB_i^\dg)$ also holds for any $i$.
Now, thanks once also to HPD (see \cite[Thm. 5.3]{kuznetsov:hpd}),
there exists also a triangulated category $\bbC_L$ and semi-orthogonal decompositions
$$
\begin{array}{rl}
\perf(X_L) =& \langle \bbC_L, \bbA_{\mathrm{dim}(L)}(1), \ldots, \bbA_{n}(n-\mathrm{dim}(L)) \rangle \\
\perf(Y_L) =& \langle \bbB_{m}(\mathrm{dim}(L^\perp) - m), \ldots, \bbB_{\mathrm{dim}(L^\perp)}(-1), \bbC_L \rangle\,.
\end{array}
$$
Moreover, the composed functor $\perf(X_L) \to \bbC_L \to \perf(Y_L)$ is of Fourier-Mukai type. As a conclusion, since $X_L$ and $Y_L$ are smooth,
we can apply Theorem \ref{thm:main1} and obtain the following chain of equivalences
$$V(X_L) \Leftrightarrow V_{NC}(\perf_\dg(X_L)) \Leftrightarrow V_{NC}(\bbC_L^\dg) \Leftrightarrow V_{NC}(\perf_\dg(Y_L)) \Leftrightarrow V(Y_L)\,.$$
This achieves the proof.
\section{Proof of Theorem~\ref{thm:main-moish-ex}}
Thanks to the work of Cossec \cite{cossec}, the conic bundle $Y \to \bbP^2$ has a natural structure of
quartic double solid $Y \to \bbP^3$ ramified along a quartic symmetroid $D$. Via the natural involution
on the resolution of singularities of $D$, one hence obtains an Enriques surface $S$; consult \cite[\S 3]{cossec} for details. As proved by Zube in \cite[\S 5]{zube}, one has moreover a semi-orthogonal decomposition
$$\perf(S) = \langle \cT_S, \cE_1, \ldots, \cE_{10}\rangle,$$
with the $\cE_i$'s exceptional objects. Let us denote by $\cT_S^\dg$ the dg enhancement of $\cT_S$ induced from $\perf_\dg(S)$. Thanks to Lemma~\ref{lem:smoothness}, $\cT_S^\dg$ is smooth and proper. Hence, since $S$ is a surface, an argument similar to the one of the proof of Theorem~\ref{thm:main2}(i) shows us that conjecture $V_{NC}(\cT_S^\dg)$ holds.

Now, recall from Ingalls and Kuznetsov \cite[\S 5.5]{ingalls-kuznetsov} the construction of the Fourier-Mukai functor $\Phi:\perf(S) \to \perf(X)$ whose restriction to $\cT_S$ is fully-faithful. As proved in \cite[Prop. 3.8 and Thm. 4.3]{ingalls-kuznetsov}, one has a semi-orthogonal decomposition
$$\perf(X) = \langle \Phi(\cT_S), \cE'_1, \cE'_2 \rangle,$$
with the $\cE'_i$'s exceptional objects. As a consequence, we obtain the equivalence
\begin{equation}\label{eq:conjecture}
V_{NC}(\perf_\dg(X)) \Leftrightarrow V_{NC}(\Phi(\cT_S)^\dg)\,,
\end{equation}
where $\Phi(\cT_S)^\dg$ stands for the dg enhancement of $\Phi(\cT_S)$ induced from $\perf_\dg(X)$. Since the kernel $\cK$ of the above Fourier-Mukai functor $\Phi$ gives rise to a Morita equivalence $\Phi_\dg^\cK: \cT_S^\dg \to \Phi(\cT_S)^\dg$, we conclude that the conjecture \eqref{eq:conjecture} also holds. This achieves the proof.

\subsection*{Acknowledgments:} 
The authors are grateful to Bruno Kahn and Claire Voisin for useful comments and answers. G.~Tabuada was partially supported by a NSF CAREER Award.

\end{document}